\documentclass[a4paper,reqno,final]{amsart}

\usepackage{cite}
\usepackage{enumerate}
\usepackage{nicefrac}
\usepackage{color}
\usepackage{xspace}
\usepackage{comment}
\usepackage{mathtools}

\usepackage{amsmath,amsthm,amssymb}
\usepackage[foot]{amsaddr}
\newtheorem{theorem}{Theorem}[section]
\newtheorem{lemma}[theorem]{Lemma}
\newtheorem{proposition}[theorem]{Proposition}

\theoremstyle{definition}

\newtheorem{example}[theorem]{Example}

\theoremstyle{remark}
\newtheorem{remark}[theorem]{Remark}

\usepackage{libertinus}

\newcommand{\set}[1]{\mathcal{#1}}
\newcommand{\spc}[1]{\mathbb{#1}}

\renewcommand{\ge}{\geqslant}

\renewcommand{\le}{\leqslant}

\newcommand{\M}{\set{M}}
\newcommand{\K}{\set{A}(x_0)}
\newcommand{\F}{\set{F}}
\newcommand{\G}{\Lambda}

\newcommand{\X}{\spc{X}}

\newcommand{\x}{x_\infty}

\newcommand{\A}{\mathrm{D}}

\newcommand{\norm}[1]{\|#1\|}

\newcommand{\wc}{\rightharpoonup}

\renewcommand{\F}[1]{\mathrm{Fix}(#1)}

\newcommand{\seq}{\Sigma}
\newcommand{\seqt}{\Sigma_\eps}
\newcommand{\pS}{\psi_{\Sigma}}

\newcommand{\li}{\lambda}

\newcommand{\NE}{{\rm($\sharp$)}\xspace}
\newcommand{\WC}{{\rm($\flat$)}\xspace}

\newcommand{\MM}{{\rm(M)}\xspace}
\renewcommand{\SS}{{\rm($\star$)}\xspace}

\newcommand{\eps}{\tau}

\begin{document}

\title[Weakly convergent fixed point iterations]{Weakly convergent fixed point iterations\\ for weakly sequentially non-expansive mappings}

\author{Thomas P.~Wihler}
\address{Mathematics Institute, University of Bern, CH-3012 Bern}
\email{thomas.wihler@unibe.ch}

\subjclass[2010]{47H10,47H09,47J25,65J15}

\keywords{Fixed point theorems, fixed point iterations, weak convergence, Opial spaces, weak sequential non-expansiveness}

\thanks{%
      Research supported by the Swiss National Science Foundation,
      Grant No. 200021\_212868}

\date{\today}

\begin{abstract}
Fixed point iterations are a fundamental tool in numerical analysis and scientific computing for the approximation of solutions to nonlinear problems. Their convergence is often established via the Banach fixed point theorem, provided that a suitable contraction property can be verified. However, such conditions are typically too restrictive for more complex nonlinear equations that lack key structural features such as monotonicity or convexity.
In this paper, we develop a general framework for the weak convergence of fixed point iterations based on asymptotic bounds. In particular, we introduce and exploit a \emph{weak sequential non-expansiveness} property, which is significantly weaker than the global Lipschitz assumptions commonly employed in this context. This approach permits us to extend classical convergence results to a broader class of mappings in general (reflexive) Opial spaces, without relying on additional geometric assumptions such as uniform convexity.
\end{abstract}

\maketitle


\section{Fixed point iterations in Opial spaces}

In practical applications, solutions of nonlinear problems (such as algebraic or differential equations) are often approximated with the aid of suitable fixed point iteration schemes, viz.
\begin{equation}\label{eq:fp}
x_{n+1}=f(x_n),\qquad n\ge 0,
\end{equation}
where $f:\,\M\to\X$ is a mapping on a (non-empty)
subset $\M$ of a (real) Banach space~$\X$, equipped with a norm $\norm{\cdot}$, and $x_0\in\M$ is a starting value so that $x_n\in\M$ for all $n\ge 0$ (i.e., $f$ can be iterated for $x_0$). The focus of this contribution is to study  the weak convergence of the above iteration to a fixed point of $f$, i.e., to the set
$
\F{f}=\{x\in\M:\,f(x)=x\}
$
of all fixed points of $f$ in $\M$, under \emph{asymptotic conditions}. To simplify notation, for a given initial value $x_0\in\M$, we denote by
\begin{equation}\label{eq:seq}
\seq(x_0)=\{x_n\}_{n\ge 0}=\left\{f^n(x_0)\right\}_{n\ge 0}
\end{equation}
the resulting sequence generated from~\eqref{eq:fp}, which, as indicated before, is assumed to belong entirely to the set~$\M$.\medskip

From a theoretical perspective, a convergence analysis for~\eqref{eq:fp} that yields practically useful results involves three key aspects. These will be outlined in the sequel in sufficient detail to guide interested non-experts into the subject and to briefly indicate a few historical landmarks.

\subsection*{I. Existence of weakly convergent subsequences} 
The sequence $\seq(x_0)\subset\M$ from \eqref{eq:seq} needs to contain a \emph{weakly convergent subsequence}, with a weak limit $x_\infty\in\M$, i.e., $x_{n_k}\wc x_\infty$ as $k\to\infty$. Evidently, this is guaranteed if the set $\M$ is weakly sequentially compact. In the present paper, by referring to the Eberlein-\v{S}mulian theorem, this property is established by assuming that
$\X$ is a \emph{reflexive Banach space}.
Indeed, under this condition, any bounded sequence has a weakly convergent subsequence; moreover, if $\M$ is weakly sequentially closed then $x_\infty\in\M$. For practical reasons, we consider the following property:
\begin{enumerate}[\MM]
\item The set $\M$ is non-empty, \emph{bounded}, \emph{(norm) closed}, and \emph{convex}.
\end{enumerate}
This is a useful assumption since, by Mazur’s lemma (which implies that any closed and convex set is weakly sequentially closed), we observe the following fact:
\begin{enumerate}[\SS]
\item Any set $\M$ of type~\MM in a reflexive Banach space is weakly sequentially compact.
\end{enumerate}

\subsection*{II. Full convergence} 
A mechanism that yields the weak convergence of the \emph{full} sequence~$\seq(x_0)$ to a limit $\x$, i.e. $x_n\wc \x$ as $n\to\infty$, is required. In order to address this point, in the present article, we assume that $\X$ is a so-called \emph{Opial space}; by definition, in such spaces, any weakly convergent sequence $x_n\wc \x$ satisfies the \emph{Opial condition}:
\begin{equation}\label{eq:Opial}
\liminf_{n\to\infty}\norm{x_n-\x}<\liminf_{n\to\infty}\norm{x_n-z}\qquad\forall z\neq x_\infty.
\end{equation}
This property, which is an important tool in proving weak convergence of (entire) sequences, was discovered by Opial in his seminal paper~\cite{Opial:1967}; we will expand on this approach in \S\ref{sc:fp} of the current work. In particular, Opial proved that~\eqref{eq:Opial} holds true in any Hilbert space, cf.~\cite[Lem.~1]{Opial:1967}, as well as in uniformly convex Banach spaces with a weakly continuous duality mapping; this justifies the relevance of~\eqref{eq:Opial} in many practical applications. We also mentioned that, for any separable Banach space, there is an equivalent norm such that~\eqref{eq:Opial} holds true, see~\cite[Thm.~1]{Dulst:82}. For a characterization of Opial spaces in the context of certain classes of Banach spaces,
we further refer to, e.g., in~\cite[Thm.~2]{Sims:1985} or~\cite[\S3]{SoTi:2016}.

\subsection*{III. Convergence of fixed point iterations} 
Finally, the question of whether the (weak) limit of the sequence $\seq(x_0)$ is a fixed point of $f$ needs to be answered affirmatively. This usually requires an appropriate continuity assumption on the mapping $f$, which is typically expressed in terms of global Lipschitz-type bounds in the literature. For instance, in the framework of the celebrated Banach fixed point theorem, a uniform contraction bound
\[
\norm{f(x)-f(y)}\le\kappa\norm{x-y}\qquad\forall x,y\in\M,
\]
with a constant $\kappa\in(0,1)$, yields the strong convergence (i.e. in norm) of the  iteration~\eqref{eq:fp} to the unique fixed point of~$f$ in $\M$. Various generalizations are available in the literature, including, e.g., the Boyd-Wong~\cite{BoydWong:1969} or Meir-Keeler~\cite{MeirKeeler:1969} contraction criteria. Moreover, under suitable compactness assumptions (on either the set $\M$ or the mapping~$f$), Edelstein's theorem~\cite[Thm.~1]{Edelstein:1962} states that the contraction framework can be extended to so-called \emph{weak contractions}, viz.
\[
\norm{f(x)-f(y)}<\norm{x-y}\qquad\forall x,y\in\M,\quad x\neq y,
\]
where strong convergence to the (unique) fixed point of $f$ can be retained. We refer to the monograph~\cite{Pata:2010} for further details on the subject.

A breaking point with regard to the existence and uniqueness of fixed points, and to the convergence of~\eqref{eq:fp}, occurs when $f$ is merely uniformly 1-Lipschitz: 
\begin{equation}\label{eq:ne'}
\norm{f(x)-f(y)}\le\norm{x-y}\qquad\forall x,y\in\M;
\end{equation}
such mappings are called \emph{non-expansive}. In this situation, the fixed point set of $f$ may no longer be a singleton (as can be seen, e.g., if $f$ is the identity map), and strong convergence of the fixed point iteration is generally unattainable. Yet, under the condition~\eqref{eq:ne'}, the famous Browder-G\"ohde-Kirk theorem (see, e.g., \cite[Thm.~10.A]{Zeidler:I}) guarantees the existence of fixed points if the underlying Banach space is uniformly convex, and the set $\M$ is of type~\MM. Furthermore, within this framework, under the Opial condition~\eqref{eq:Opial}, for any $\eps\in(0,1)$, the \emph{modified successive iteration} given by
\begin{equation}\label{eq:mod}
x_{n+1}=\eps x_n+(1-\eps)f(x_n),\qquad n\ge 0,
\end{equation}
converges \emph{weakly} to a fixed point of~$f$; see~\cite[Thm.~3]{Opial:1967}. Under the additional assumption that the sequence generated by~\eqref{eq:mod} is \emph{asymptotically regular}, i.e.
\begin{equation}\label{eq:ar}
\norm{x_{n+1}-x_n}\to 0\qquad\text{as }n\to\infty,
\end{equation}
the weak convergence of~\eqref{eq:mod} is retained for $\eps=0$; this is known as \emph{Opial's theorem}, see~\cite[Thm.~2]{Opial:1967}. Conditions that imply the asymptotic regularity~\eqref{eq:ar} in the context of fixed point iterations have been studied extensively, see, e.g., the works~\cite{BaillonBruckReich:1978,BruckReich:1977,Ishikawa:1976,ReichShafrir:1987}.

In the present paper, rather than employing \emph{global and uniform} Lipschitz bounds as in the aforementioned works, we make use of significantly weaker, asymptotic conditions. We begin with the following definition:
\begin{enumerate}[\NE]
\item The mapping $f$ is called \emph{weakly sequentially non-expansive} at a point $y\in\M$ if, for any sequence $\{y_n\}_n\subset\M$ that converges weakly to $y$, i.e. $y_n\wc y$ as $n\to\infty$, it holds that
\begin{equation}\label{eq:ne}
\liminf_{n\to\infty}\norm{f(y_n)-f(y)}\le\liminf_{n\to\infty}\norm{y_n-y}.
\end{equation}
We say that $f$ satisfies \NE on $\M$ if~\eqref{eq:ne} holds at any point $y\in\M$.
\end{enumerate}
Clearly, \NE holds true for any global 1-Lipschitz mapping, but is much weaker than~\eqref{eq:ne'}. It is also weaker than weak-to-strong continuity; indeed, \NE only controls the asymptotic behavior of $f$ along weakly convergent sequences in a one-sided (liminf) sense. In particular, the bound \NE does \emph{not} even force $f$ to be continuous. To illustrate this, on a Banach space $\X$, define the mapping
\[
f(x)=\begin{cases}
0&\text{for }x=0,\\
\norm{2x}^{-1}x&\text{otherwise},
\end{cases}
\]
which is discontinuous at $0$,  and consider a sequence $\{y_n\}_{n\ge 0}\subset\X\setminus\{0\}$ with $\norm{y_n}=1$ for all $n\in\mathbb{N}$, and with $y_n\wc y=0$; then,
\[
\liminf_{n\to\infty}\norm{f(y_n)-f(0)}=\nicefrac12\qquad\text{but}\qquad\liminf_{n\to\infty}\norm{y_n-0}=1,
\]
which shows that~\eqref{eq:ne} holds despite the discontinuity of~$f$.

\begin{example}
To give an example for which \NE is fulfilled, for any $y\in\M$, suppose that the mapping $f:\,\M\to\X$ can be approximated by a (not necessarily linear) operator $\A_y:\,\X\to\X$ in the sense that
\[
\norm{f(y+h)-f(y)-\A_y(h)}\le (1-\alpha_y)\norm{h}+\omega_y(h),
\]
for any $h\in\M$ so that $y+h\in\M$; here, we impose the bound 
\[
\norm{\A_y(h)}\le \alpha_y\norm{h}+\rho_y(h),\qquad h\in\X,
\]
for some ($y$-dependent) constant $\alpha_y\le 1$; for the two remainder terms $\omega_y$ and $\rho_y$, we assume that $\omega_y(h)\to0$ and $\rho_y(h)\to0$ as $h\wc 0$. Then, for a weakly convergent sequence $y_n\wc y$ in $\M$, we have
\begin{align*}
\norm{f(y_n)-f(y)}
&\le (1-\alpha_y)\norm{y_n-y}+\norm{\A_y(y_n-y)}+\omega_y(y_n-y)\\
&\le \norm{y_n-y}+\rho_y(y_n-y)+\omega_y(y_n-y),
\end{align*}
for any $n\ge 0$. For $n\to\infty$ the last two terms on the right-hand side tend to 0, whence we derive~\eqref{eq:ne}.
\end{example}

\subsection*{Goal of the paper}
In addition to the above setting {\bfseries I.}--{\bfseries III.}, we note that the existence of fixed points and the weak convergence of the fixed point iteration~\eqref{eq:fp} may require further structural conditions on~$f$.
The aim of this paper is to replace global Lipschitz assumptions by asymptotic conditions, thereby extending classical convergence results to a substantially broader setting.
%
In particular, we obtain new results on the existence of fixed points (Theorem~\ref{thm:ex}), as well as on the weak convergence of the iterative schemes~\eqref{eq:fp} and~\eqref{eq:mod}; see Theorems~\ref{thm:main1} and~\ref{thm:main2}, respectively. We emphasize that our analysis applies to general reflexive Opial spaces and does \emph{not} rely on additional assumptions such as uniform convexity.


\section{Existence of fixed points in Opial spaces}\label{sc:ex}

To establish the existence of fixed points, in addition to \NE, we impose a condition on the mapping~$f$ that can be regarded, in some sense, as a weak, asymptotic converse to non-expansiveness. It involves the quantity
\begin{equation}\label{eq:dM}
d_{\M}:=\sup_{v,w\in\M}\norm{v-w},
\end{equation}
which signifies the diameter of the set~$\M$ (note that $d_\M<\infty$ under condition~\MM). 
Throughout, we assume that $d_\M>0$.
\begin{enumerate}[\WC]
\item There is a constant $\delta\in(0,1)$ (depending on $f$) such that, whenever a sequence $\{y_n\}_{n}\subset\M$ satisfies 
\begin{equation}\label{eq:wc}
0
<\limsup_{n\to\infty}\norm{y_{n+1}-y_n}
<\limsup_{n\to\infty}\norm{f(y_{n+1})-f(y_{n})},
\end{equation}
then it holds
\begin{equation}\label{eq:auxb}
\limsup_{n\to\infty}\norm{y_{n+1}-f(y_n)}>\delta d_\M.
\end{equation}
\end{enumerate}
The meaning of condition~\WC is that it excludes sequences along which $f$ expands distances asymptotically without forcing a corresponding deviation from the fixed point structure.

\begin{example}\label{ex:neq}
The second bound in~\eqref{eq:wc} is never fulfilled if $f$ is non-expansive, cf.~\eqref{eq:ne'}, whence we conclude that \WC is always satisfied for such mappings.
More generally, we note that if there exists $\epsilon>0$ such that $f$ is non-expansive on the set 
\[
\set{B}_\epsilon=\{(x,y)\in\M:\,\norm{y-f(x)}\le\epsilon\}\subset\M,
\] 
i.e.
\begin{equation}\label{eq:neloc}
\norm{f(x)-f(y)}\le\norm{x-y}
\qquad \forall\,x,y\in\set{B}_\epsilon,
\end{equation}
then we see that condition~\WC is fulfilled for any $\delta\in(0,1)$ with $\delta<\nicefrac{\epsilon}{d_{\M}}$. Indeed, suppose that~\eqref{eq:wc} holds for some sequence $\{y_n\}_n\in\M$, then we can find a subsequence $\{y_{n_k}\}_k$ and $\theta>0$ such that, in particular,
\[
\theta+\norm{y_{n_k+1}-y_{n_k}}
<\norm{f(y_{n_k+1})-f(y_{n_k})},
\]
for any $k\ge k_0$, with $k_0\in\mathbb{N}$ sufficiently large. From~\eqref{eq:neloc}, we infer that 
\[
\norm{y_{n_k+1}-f(y_{n_k})}>\epsilon\qquad\forall k\ge k_0,
\] 
and thus $\limsup_{n\to\infty}\norm{y_{n+1}-f(y_{n})}\ge\epsilon>\delta d_\M$, which gives~\eqref{eq:auxb}.

\end{example}

\begin{theorem}\label{thm:ex}
Consider a subset $\M$ of type \MM of a reflexive Opial space~$\X$. Then, any self-map $f:\,\M\to\M$ that satisfies the properties~\NE and \WC has a fixed point. 
\end{theorem}

\begin{proof}
Consider a sequence $\{\epsilon_n\}_n\subset(0,\delta)$ with $\epsilon_n\to 0^+$ as $n\to\infty$, where $\delta\in(0,1)$ is the constant from~\WC, and a point $z\in\M$. Then, for given $n\in\mathbb{N}$, exploiting the convexity of $\M$, we define the iteration 
\[
x_{k+1}=\phi_n(x_k):=(1-\epsilon_n)f(x_k)+\epsilon_n z,\qquad k\ge 0,
\]
with a starting point $x_0\in\M$. For each $k\ge 0$, note that
\[
\norm{x_{k+1}-f(x_k)}=\epsilon_n\norm{z-f(x_k)}\le
\epsilon_n d_\M\le
\delta d_\M,
\]
and
\[
\norm{x_{k+2}-x_{k+1}}=(1-\epsilon_n)\norm{f(x_{k+1})-f(x_k)}.
\]
Thus, we have $\limsup_{k\to\infty}\norm{x_{k+1}-f(x_k)}\le\delta d_\M$, which is the converse of the bound \eqref{eq:auxb}, as well as
\[
\li\le(1-\epsilon_n)\limsup_{k\to\infty}\norm{f(x_{k+1})-f(x_k)},
\]
where we let $\li:=\limsup_{k\to\infty}\norm{x_{k+1}-x_{k}}$.
Then, we either have 
\[
\limsup_{k\to\infty}\norm{f(x_{k+1})-f(x_k)}=0,
\] 
and therefore $\lambda=0$, or otherwise 
\[
\li < \limsup_{k\to\infty}\norm{f(x_{k+1})-f(x_k)},
\]
which is the second bound in~\eqref{eq:wc}. Thus, if the second inequality in~\eqref{eq:wc} holds, then the first one must fail; hence in all cases we obtain $\lambda=0$. Therefore, we deduce that
\[
\li=\limsup_{k\to\infty}\norm{x_{k+1}-x_k}=\lim_{k\to\infty}\norm{x_{k+1}-x_k}=0,
\] 
i.e., the sequence $\{x_k\}_k$ is asymptotically regular. Moreover,
involving the weak sequential compactness of $\M$, cf.~\SS, we infer that there is a weakly convergent subsequence $x_{k_\ell}\wc\xi_n\in\M$.
Thus, using asymptotical regularity, we obtain
\begin{align*}
\liminf_{\ell\to\infty}\norm{\phi_n(\xi_n)-x_{k_\ell}}
&=\liminf_{\ell\to\infty}\norm{\phi_n(\xi_n)-x_{k_\ell+1}}\\
&=\liminf_{\ell\to\infty}\norm{\phi_n(\xi_n)-\phi_n(x_{k_\ell})}\\
&=(1-\epsilon_n)\liminf_{\ell\to\infty}\norm{f(\xi_n)-f(x_{k_\ell})}.
\end{align*}
Applying~\eqref{eq:ne}, it follows that
\[
\liminf_{\ell\to\infty}\norm{\phi_n(\xi_n)-x_{k_\ell}}
\le(1-\epsilon_n)\liminf_{\ell\to\infty}\norm{\xi_n-x_{k_\ell}}
\le\liminf_{\ell\to\infty}\norm{\xi_n-x_{k_\ell}},
\]
from which we conclude that $\phi_n(\xi_n)=\xi_n$ by the Opial condition~\eqref{eq:Opial}, i.e. $\xi_n$ is a fixed point of $\phi_n$. Consequently, we have that 
\begin{equation}\label{eq:aux2026}
\norm{f(\xi_n)-\xi_n}=\norm{f(\xi_n)-\phi_n(\xi_n)}=\epsilon_n\norm{f(\xi_n)-z}
\le\epsilon_nd_\M\xrightarrow{n\to\infty} 0,
\end{equation}
by the boundedness of $\M$ and the fact that $\epsilon_n\to0^+$ as $n\to\infty$. Again, in light of property~(M), there is a weakly convergent subsequence $\xi_{n_m}\wc \xi^\star\in\M$ as $m\to\infty$. Then, with the aid of~\eqref{eq:aux2026} and of~\eqref{eq:ne}, we arrive at
\begin{align*}
\liminf_{m\to\infty}\norm{f(\xi^\star)-\xi_{n_m}}
=\liminf_{m\to\infty}\norm{f(\xi^\star)-f(\xi_{n_m})}
\le\liminf_{m\to\infty}\norm{\xi^\star-\xi_{n_m}},
\end{align*}
which, by Opial's condition~\eqref{eq:Opial}, yields $f(\xi^\star)=\xi^\star$.
\end{proof}


\section{Weak convergence of fixed point iterations}

\subsection{Weak convergence of full sequences}\label{sc:fp}

For a given sequence $\seq=\{x_n\}_{n}\subset\X$ in a reflexive Opial space $\X$, the set defined by
\begin{equation}\label{eq:L}
\G(\seq):=\left\{z\in\X:\,\lim_{n\to\infty}\norm{x_n-z}\text{ exists}\right\}
\end{equation}
will play a central role for the analysis to be developed in the sequel. 

\begin{remark}\label{rem:subseq}
We collect a few observations about the above set $\G(\seq)$.
\begin{enumerate}[(a)]
\item Crucially for this work, if $\G(\seq)\neq\emptyset$, then we conclude that $\seq$ is necessarily bounded; in particular, since the Banach space $\X$ is assumed reflexive, the sequence $\seq$ has a weakly convergent subsequence.

\item We note, however, that $\G(\seq)$ may be empty even when $\seq$ converges weakly. Indeed, consider the sequence 
\[
x_n=\begin{cases}
e_n&\text{if $n$ is even},\\
2e_n&\text{if $n$ is odd},
\end{cases}
\]
where $\{e_n\}_{n\ge1}$ is the canonical basis in the Hilbert space $\ell_2(\mathbb{R})$. It holds that $x_n\wc 0$ as $n\to\infty$. Furthermore, for any $z\in\ell_2(\mathbb{R})$, we have that
\[
\norm{x_n-z}^2
=\norm{x_n}^2-2(x_n,z)+\norm{z}^2
=\norm{x_n}^2+(z-2x_n,z),
\]
where $(\cdot,\cdot)$ denotes the inner product in~$\ell_2(\mathbb{R})$;
on the right-hand side, the norm $\norm{x_n}^2$ oscillates between $1$ and $2$, and does hence not converge, while the remaining term tends to $\norm{z}^2$ as $n\to\infty$. Consequently, $z\not\in\G(\seq)$ for any $z\in\ell_2(\mathbb{R})$.
\end{enumerate}
\end{remark}

If $\G(\seq)\neq\emptyset$, we introduce the function
\begin{equation*}
\pS(z):\, \G(\seq)\to[0,\infty),\qquad z\mapsto\pS(z):=\lim_{n\to\infty}\norm{x_n-z}.
\end{equation*}
We observe the following instrumental result, which illustrates the essential mechanism of the Opial condition~\eqref{eq:Opial} applied in this paper.

\begin{lemma}\label{lem:my}
If a sequence $\seq$ in a (not necessarily reflexive) Opial space $\X$ possesses a weak accumulation point $w\in\G(\seq)$, then
\begin{equation}\label{eq:pmin}
\pS(w)<\pS(z)\qquad\forall z\in\G(\seq)\setminus\{w\},
\end{equation}
i.e., $w$ is the (unique) global minimizer of~$\pS$ on $\G(\seq)$.
\end{lemma}

\begin{proof}
Consider a subsequence $\{x_{n_k}\}_k\subset\seq$ with $x_{n_k}\wc w\in\G(\seq)$ as $k\to\infty$; then, for an arbitrary point $z\in\G(\seq)$, $z\neq w$, involving the Opial property~\eqref{eq:Opial}, we see that
\[
\pS(w)=\lim_{k\to\infty}\norm{x_{n_k}-w}<\lim_{k\to\infty}\norm{x_{n_k}-z}=\pS(z),
\]
which shows~\eqref{eq:pmin}.
\end{proof}

We can now derive an important connection between the weak convergence of a sequence $\seq\subset\X$ and the set $\G(\seq)$ from~\eqref{eq:L}.

\begin{proposition}\label{prop:conv}
Consider a sequence $\seq$ in a reflexive Opial space $\X$ with $\G(\seq)\neq\emptyset$.
If all weak accumulation points of $\seq$ belong to $\G(\seq)$, then the (entire) sequence $\seq$ converges weakly.
\end{proposition}

\begin{proof}
Recalling Remark~\ref{rem:subseq}~(a), since $\G(\seq)\neq\emptyset$, we note that the sequence $\seq$ is bounded. Hence, by reflexivity of $\X$, there is a weakly convergent subsequence $x_{n_k}\wc w$ as $k\to\infty$. Employing Lemma~\ref{lem:my}, we have $\pS(w)<\pS(z)$ for any $z\in\G(\seq)$, $z\neq w$. In particular, since by assumption every weak accumulation point of $\seq$ lies in $\G(\seq)$, we conclude that there is exactly one such point. Therefore, by weak sequential compactness, the whole sequence converges weakly to $w$, that is $x_n \wc w$ as $n \to \infty$.
\end{proof}

\begin{remark} 
The hypothesis in the above Proposition~\ref{prop:conv}, whereby all weak accumulation points of $\seq$ belong to $\G(\seq)\neq\emptyset$, can be replaced by assuming that $\seq$ is bounded and that there is a point $p\in\X$ with
\[
\limsup_{n\to\infty}\norm{x_n-p}\le\liminf_{n\to\infty}\norm{x_n-w},
\]
for all weak accumulation points~$w$ of $\seq$. Indeed, suppose that $w$ is a weak accumulation point of $\seq$, with a corresponding convergent subsequence $x_{n_k}\wc w$ as $k\to\infty$. Then, Opial's condition~\eqref{eq:Opial} implies that
\[
\liminf_{k\to\infty}\norm{x_{n_k}-w}
\le\liminf_{k\to\infty}\norm{x_{n_k}-p}
\le\limsup_{n\to\infty}\norm{x_n-p}
\le\liminf_{n\to\infty}\norm{x_n-w},
\]
and thus
\[
\liminf_{n\to\infty}\norm{x_n-w}=\liminf_{k\to\infty}\norm{x_{n_k}-w}=\lim_{n\to\infty}\norm{x_n-p},
\]
which yields~$p\in\G(\seq)$. Furthermore, again by~\eqref{eq:Opial}, the second of the above equalities can only hold for $w=p$, whence we deduce that there is exactly one weak accumulation point of $\seq$, and thus $x_n\wc p$ as $n\to\infty$.
\end{remark}

\begin{remark}
We note that a sequence $\seq$ may have two different weak accumulation points, say $w_1$ and $w_2$, such that $w_1\in\G(\seq)$ and $w_2\not\in\G(\seq)$. Indeed, consider again the canonical basis $\{e_n\}_{n\ge 1}\subset\ell_2(\mathbb{R})$ from Remark~\ref{rem:subseq}~(b), and define a sequence $\seq=\{x_n\}_{n\ge 1}$ by
\[
x_n=\begin{cases}
e_n&\text{if $n$ is even},\\
e_1&\text{if $n$ is odd}.
\end{cases}
\]
Then, $\norm{x_n}=1$ for each $n\ge 1$, and $e_{2n}\wc 0$ as $n\to\infty$. Hence, $w_1=0$ is a weak accumulation point that belongs to $\G(\seq)$. Furthermore, there is a second weak accumulation point $w_2=e_1$, for which
\[
\norm{w_2-x_n}=\begin{cases}
\sqrt2&\text{if $n\ge 2$ is even},\\
0&\text{if $n$ is odd}
\end{cases}
\]
does not converge. This example demonstrates that the validity of the above Proposition~\ref{prop:conv} requires more than the inclusion of a single weak accumulation point in~$\G(\seq)$ when multiple weak accumulation points exist.
\end{remark}


\subsection{Weak convergence to fixed points}

We establish the weak convergence of the fixed point iteration~\eqref{eq:fp}, for a suitable initial point $x_0\in\M$, by proving a weaker version of Opial's original theorem that is based on condition~\NE rather than on global non-expansiveness, cf.~\eqref{eq:ne'}. 
In the sequel, we denote by $\K$ the set of 
all weak accumulation points of the sequence $\seq(x_0)$ generated by~\eqref{eq:fp} in $\M$, and recall the notion~\eqref{eq:ar} of asymptotic regularity. A key assumption in the following analysis is that $\F{f}\subset\G(\seq(x_0))$, which is satisfied, for instance, under quasi-Fejér monotonicity; see Example~\ref{ex:sF} below.

\begin{theorem}\label{thm:main1}
Let $x_0\in\M$ be a starting point for the iteration~\eqref{eq:fp} in a subset $\M$ of type~\MM of a reflexive Opial space $\X$ such that the resulting sequence $\seq(x_0)$ belongs to $\M$ and is asymptotically regular, cf.~\eqref{eq:ar}. Suppose that $f$ satisfies condition~\NE on $\K$. If\, $\F{f}\subset\G(\seq(x_0))$ then the (entire) sequence $\seq(x_0)$ converges weakly to a point $w\in\F{f}$, which satisfies 
\begin{equation}\label{eq:psix}
\pS(w)=\min_{z\in\G(\seq(x_0))}\pS(z)=\min_{z\in\F{f}}\pS(z).
\end{equation}
\end{theorem}

\begin{proof}
Using that $\X$ is reflexive, and that the set $\M$ satisfies property~\MM, we note that $\K\neq\emptyset$, cf.~\SS. Our goal is to prove that 
\begin{equation}\label{eq:prove}
\K\subset\G(\seq(x_0)).
\end{equation} 
Then, from Proposition~\ref{prop:conv}, we immediately deduce the weak convergence of the (full) sequence $x_n\wc w$ as $n\to\infty$ to a weak limit $w\in\G(\seq(x_0))\cap\M$, and $\K=\{w\}$. Moreover, by virtue of Lemma~\ref{lem:my}, we know that~$w$ is the unique minimizer of $\pS$ on $\G(\seq(x_0))$, which yields the first equality in~\eqref{eq:psix}. 

To establish the inclusion~\eqref{eq:prove}, suppose that $f(w)\neq w$. Then, employing the Opial property~\eqref{eq:Opial}, and using asymptotic regularity, we observe that
\begin{align*}
\liminf_{n\to\infty}\norm{x_{n}-w}
&<
\liminf_{n\to\infty}\norm{x_{n}-f(w)}\\
&=\liminf_{n\to\infty}\norm{x_{n+1}-f(w)}
=\liminf_{n\to\infty}\norm{f(x_n)-f(w)},
\end{align*}
which contradicts the bound~\eqref{eq:ne}, and thus, we obtain $f(w)=w$. Hence, we conclude that $\K=\{w\}\subset\F{f}\subset\G(\seq(x_0))$, which shows~\eqref{eq:prove} as well as the second equality in~\eqref{eq:psix}.
\end{proof}
 
 \begin{remark}\label{rem:arnew}
The proof of the above Theorem~\ref{thm:main1} is essentially based on~\eqref{eq:prove}. Incidentally, this inclusion could be replaced by $f(\K)\subset\G(\seq(x_0))$. Indeed, consider \emph{any} weak accumulation point $w\in\K$, and a corresponding subsequence $x_{n_k}\wc w$ as $k\to\infty$. Suppose that $f(w)\neq w$, then by the Opial condition~\eqref{eq:Opial} it holds that
\begin{equation*}
\liminf_{k\to\infty}\norm{w-x_{n_k}}
<\liminf_{k\to\infty}\norm{f(w)-x_{n_k}}.
\end{equation*}
Since $f(w)\in\G(\seq(x_0))$, we infer that
\begin{equation*}
\liminf_{k\to\infty}\norm{w-x_{n_k}}
<\lim_{k\to\infty}\norm{f(w)-x_{n_k+1}}
=\lim_{k\to\infty}\norm{f(w)-f(x_{n_k})}
,
\end{equation*}
which causes a contradiction to the weak sequential non-expansiveness~\NE. Hence, we see that 
$f(w)=w$, and thus $f(\K)=\K$; from this, in turn, we recover~\eqref{eq:prove}.
\end{remark}

\begin{example}\label{ex:sF}
The assumption $\F{f}\subset\G(\seq(x_0))$ in the above Theorem~\ref{thm:main1} may be ensured, for instance, if the sequence $\seq(x_0)$ from~\eqref{eq:seq} is \emph{quasi-Fej\'er monotone}, i.e. for any $y\in\F{f}$, it holds that
\begin{equation}\label{eq:qF}
\norm{x_{n+1}-y}\le\norm{x_n-y}+\eta_n(y),
\end{equation}
for any $n\in\mathbb{N}$ large enough, where we require $\eta_n(y)\ge 0$ for all $n$, and $\sum_{k=1}^\infty\eta_k(y)<\infty$. Then, the sequence given by $a_n:=\norm{x_n-y}+\sum_{k=n}^\infty\eta_k(y)$, $n\ge 0$, is monotone decreasing, and thus convergent. This, in turn, implies that $y\in\G(\seq(x_0))$, and therefore, we deduce that $\F{f}\subset\G(\seq(x_0))$.
\end{example}

\section{Application to modified successive iterations}

For a mapping $f:\M\to\X$ on a convex subset $\M\subset\X$, and a parameter $\eps\in(0,1)$, we consider the modified successive iteration
\begin{subequations}\label{eq:smi}
\begin{equation}\label{eq:smi1}
x_{n+1}=g_\eps(x_n),\qquad n\ge 0,
\end{equation}
where we define the convex combination
\begin{equation}\label{eq:smi2}
g_\eps(x):=\eps x+(1-\eps)f(x),\qquad x\in\M,
\end{equation}
\end{subequations}
cf.~\eqref{eq:mod}; by the convexity of $\M$, we note that $g_\eps(x_n)\in\M$ as long as $f(x_n)\in\M$ (which is clearly true for any $n\ge0$, for instance, if $f(\M)\subset\M$). In line with our notation~\eqref{eq:seq}, for a starting point $x_0\in\M$, we write
\begin{equation}\label{eq:seqt}
\seqt(x_0)=\{x_n\}_{n\ge 0}=\left\{g_\eps^n(x_0)\right\}_{n\ge 0}
\end{equation}
to signify the sequence generated by~\eqref{eq:smi1}. In the following result, the asymptotic regularity required in  Theorem~\ref{thm:main2} can be replaced by condition~\WC.

\begin{theorem}\label{thm:main2}
On a subset $\M$ of an Opial space $\X$, which satisfies property {\rm(M)}, consider the modified successive iteration~\eqref{eq:smi}. Suppose that the underlying mapping $f:\,\M\to\X$ in~\eqref{eq:smi2} satisfies \WC for some $\delta\in(0,1)$. For any $\eps\in(0,\delta]$, if $g_\eps$ from~\eqref{eq:smi2} is weakly sequentially non-expansive on $\M$, cf.~\NE, and the sequence $\seqt(x_0)$ from~\eqref{eq:seqt} belongs to $\M$ and fulfills $\F{f}\subset\G(\seqt(x_0))$, then the (entire) sequence $\seqt(x_0)$ converges weakly to a fixed point of $f$. 
\end{theorem}

The proof of Theorem~\ref{thm:main2} requires an auxiliary result on asymptotic regularity inspired by ~\cite[Lem.~2]{Ishikawa:1976}, which was originally established for non-expansive mappings~\eqref{eq:ne'}.

\begin{lemma}\label{lem:Ish}
For any $\eps\in(0,1)$ and a starting point $x_0\in\M$, suppose that the sequence $\seqt(x_0)$ from~\eqref{eq:seqt} satisfies
\begin{equation}\label{eq:limsup}
\limsup_{n\to\infty}\norm{f(x_{n+1})-f(x_{n})}\le\limsup_{n\to\infty}\norm{x_{n+1}-x_n}.
\end{equation}
Then, the sequence $\seqt(x_0)$ is asymptotically regular, cf.~\eqref{eq:ar}.
\end{lemma}

\begin{proof}
Suppose that the given sequence is \emph{not} asymptotically regular, i.e.
\begin{align*}
\lambda:=\limsup_{n\to\infty}\norm{x_{n+1}-x_{n}}>0.
\end{align*}
Then, we have that
\begin{equation}\label{eq:la}
\lambda
=\limsup_{n\to\infty}\norm{g_\eps(x_{n})-x_{n}}
=(1-\eps)\limsup_{n\to\infty}\norm{f(x_n)-x_n}.
\end{equation}
Furthermore, for each $n\ge 0$, letting $y_n:=f(x_n)-x_n$, 
we use~\eqref{eq:smi} to obtain
\begin{align*}
y_{n+1}-\eps y_n
&=
f(x_{n+1})-x_{n+1}-\eps(f(x_n)-x_n)
=f(x_{n+1})-f(x_{n}).
\end{align*}
Then, exploiting~\eqref{eq:limsup}, we arrive at
\begin{subequations}\label{eq:y}
\begin{align}\label{eq:y1}
\limsup_{n\to\infty}\norm{y_{n+1}-\eps y_n}
\le\lambda=r(1-\eps),
\end{align}
where we let
\begin{align}\label{eq:y2}
r:=\limsup_{n\to\infty}\norm{y_n}=\nicefrac{\lambda}{(1-\eps)}>0,
\end{align}
\end{subequations}
cf.~\eqref{eq:la}. Following the proof of~\cite[Lem.~2]{Ishikawa:1976}, for any $n\ge 1$ and $\ell\ge 2$, we notice that
\begin{align*}
x_{n+\ell+1}-x_{n+1}
&=\sum_{k=1}^\ell(x_{n+k+1}-x_{n+k})
=(1-\eps)\sum_{k=1}^\ell y_{n+k},
\end{align*}
which can be expanded further to the identity
\begin{align*}
x_{n+\ell+1}-x_{n+1}
&=\eps\left(\eps^{-\ell}-1\right)y_{n+\ell}-\sum_{k=1}^{\ell-1}\left(\eps^{-k}-1\right)(y_{n+k+1}-\eps y_{n+k}).
\end{align*}
Thus,
\begin{align}\label{eq:Ish}
d_{\M}
\ge
\eps\left(\eps^{-\ell}-1\right)\norm{y_{n+\ell}}-\sum_{k=1}^{\ell-1}\left(\eps^{-k}-1\right)\norm{y_{n+k+1}-\eps y_{n+k}},
\end{align}
with~$d_\M$ from~\eqref{eq:dM}. Now choose a (finite) integer $\ell\ge 2$ such that 
\begin{equation}\label{eq:dM1}
r(1-\tau)(\ell-1)\le d_\M+1 \le r(1-\tau)\ell.
\end{equation}
From~\eqref{eq:y}, for any $\epsilon>0$ with
\begin{equation}\label{eq:epso}
\epsilon\exp\left(r^{-1}\tau^{-1}(d_\M+1)\right)<\nicefrac{1}{2r},
\end{equation}
we can find $n\in\mathbb{N}$ sufficiently large such that $\norm{y_{n+\ell}}\ge (1-\epsilon)r$, and
\begin{align*}
\norm{y_{n+k+1}-\eps y_{n+k}}\le (1+\epsilon)(1-\eps)r
\qquad\forall k\ge n.
\end{align*}
Then, from~\eqref{eq:Ish}, we infer that
\begin{align*}
d_{\M}&\ge
\eps\left(\eps^{-\ell}-1\right)(1-\epsilon)r-(1+\epsilon)(1-\eps)r\sum_{k=1}^{\ell-1}\left(\eps^{-k}-1\right)\\
&=r(1-\eps)\ell-r\epsilon\left(2\eps\left(\eps^{-\ell}-1\right)-(1-\tau)\ell\right)
\ge d_\M+1-2r\epsilon\eps^{1-\ell}.
\end{align*}
For $s>0$, notice that $\ln(s)\le s-1$, which leads to
\[
\tau^{1-\ell}=\exp\left((\ell-1)\ln\left(\tau^{-1}\right)\right)
\le\exp\left((\ell-1)\tau^{-1}\left(1-\tau\right)\right).
\]
Incorporating~\eqref{eq:dM1}, it follows that
\[
\tau^{1-\ell}
\le\exp\left(r^{-1}\tau^{-1}(d_\M+1)\right).
\]
Finally, we conclude that
\[
d_\M\ge d_\M+1- 2r\epsilon\exp\left(r^{-1}\tau^{-1}(d_\M+1)\right),
\]
which results in a contradiction for our choice of $\epsilon$ in~\eqref{eq:epso}.
\end{proof}

\begin{proof}[Proof of Theorem~\ref{thm:main2}]
For any $\eps\in(0,1)$, observe that $\F{g_\eps}=\F{f}$, and thus we have $\F{g_\eps}\subset\G(\seqt(x_0))$. Hence, in view of Theorem~\ref{thm:main1}, it suffices to verify that the sequence $\seqt(x_0)$ is asymptotically regular. 
To this end, from~\eqref{eq:smi}, we notice that
\[
\norm{x_{n+1}-f(x_n)}
=\tau\norm{f(x_n)-x_n}\le\tau d_\M
\le\delta d_\M\qquad\forall n\ge 0,
\]
which yields the converse of the bound~\eqref{eq:auxb}. Thus, by property~\WC, if
\[
\limsup_{n\to\infty}\norm{f(x_{n+1})-f(x_{n})}
>
\limsup_{n\to\infty}\norm{x_{n+1}-x_n},
\]
then the first bound in~\eqref{eq:wc} is necessarily violated, and we deduce that $\seqt(x_0)$ is asymptotically regular. Otherwise, if 
\[
\limsup_{n\to\infty}\norm{f(x_{n+1})-f(x_{n})}
\le
\limsup_{n\to\infty}\norm{x_{n+1}-x_n},
\]
then asymptotic regularity is ensured by Lemma~\ref{lem:Ish}. 
This completes the argument.
\end{proof}

\begin{example}
Any non-expansive mapping~$f$, cf.~\eqref{eq:ne'}, satisfies the hypotheses of the above theorem (cf.~Opial's theorem~\cite[Thm.~3]{Opial:1967}); in particular, the conditions \NE and \WC immediately follow from~\eqref{eq:ne'}. In addition, for any parameter value $\eps\in(0,1)$, and for each fixed point $y\in\F{f}=\F{g_\eps}$, we note that the sequence $\seqt(x_0)$ from~\eqref{eq:seqt} is Fej\'er monotone (i.e. \eqref{eq:qF} with $\eta_n=0$ for all $n\ge 0$), and therefore $\F{f}\subset\G(\seqt(x_0))$, cf.~Example~\ref{ex:sF}.
\end{example}

\begin{remark}
Finally, we note that, in both Theorems~\ref{thm:main1} and~\ref{thm:main2}, if the image $g_\eps(\seqt(x_0))$ has a compact closure (for $\eps=0$ and $\eps\in(0,\delta]$, respectively), then the convergence of the sequence generated by~\eqref{eq:fp} and~\eqref{eq:smi1}, respectively, is strong. Clearly, this is the case, for instance, when $g_\eps$ is a compact mapping on~$\M$.
\end{remark}


\bibliographystyle{amsalpha}
\bibliography{refs}

\end{document}